\documentclass[12pt]{article}
\usepackage{tcolorbox}
\usepackage{amssymb,amsmath,bm}
\usepackage[colorlinks=true, urlcolor=blue,linkcolor=blue, citecolor=blue]{hyperref}
\usepackage{mathrsfs,verbatim}
\usepackage{caption,cleveref }
\usepackage{graphicx, float}
\usepackage{constants,color,appendix,  mathtools, tikz, xcolor}
\usepackage{enumerate}
\usepackage{bbm}
\usepackage{tkz-euclide}
\usepackage{tikz}
\usetkzobj{all}
\usetikzlibrary{trees}
\usepackage{pdfpages}
\usepackage[linesnumbered,ruled,vlined]{algorithm2e}

\usetikzlibrary{calc}
\usepackage{lipsum}

\usepackage{appendix}
\SetKwInput{KwInput}{Input}                
\SetKwInput{KwOutput}{Output}              

\begin{document}

\newcommand{\bea}{\begin{eqnarray}}
\newcommand{\ena}{\end{eqnarray}}
\newcommand{\beas}{\begin{eqnarray*}}
\newcommand{\enas}{\end{eqnarray*}}
\newcommand{\beq}{\begin{equation}}
\newcommand{\enq}{\end{equation}}
\def\qed{\hfill \mbox{\rule{0.5em}{0.5em}}}
\newcommand{\bbox}{\hfill $\Box$}
\newcommand{\ignore}[1]{}
\newcommand{\ignorex}[1]{#1}
\newcommand{\wtilde}[1]{\widetilde{#1}}
\newcommand{\mq}[1]{\mbox{#1}\quad}
\newcommand{\bs}[1]{\boldsymbol{#1}}
\newcommand{\qmq}[1]{\quad\mbox{#1}\quad}
\newcommand{\qm}[1]{\quad\mbox{#1}}
\newcommand{\nn}{\nonumber}
\newcommand{\Bvert}{\left\vert\vphantom{\frac{1}{1}}\right.}
\newcommand{\To}{\rightarrow}
\newcommand{\supp}{\mbox{supp}}
\newcommand{\law}{{\cal L}}
\newcommand{\Z}{\mathbb{Z}}
\newcommand{\mc}{\mathcal}
\newcommand{\mbf}{\mathbf}
\newcommand{\tbf}{\textbf}
\newcommand{\lp}{\left(}
\newcommand{\limm}{\lim\limits}
\newcommand{\limminf}{\liminf\limits}
\newcommand{\limmsup}{\limsup\limits}
\newcommand{\rp}{\right)}
\newcommand{\mbb}{\mathbb}
\newcommand{\rainf}{\rightarrow \infty}
\newtheorem{problem}{Problem}[section]
\newtheorem{exercise}{Exercise}[section]
\newtheorem{theorem}{Theorem}[section]
\newtheorem{corollary}{Corollary}[section]
\newtheorem{conjecture}{Conjecture}[section]
\newtheorem{proposition}{Proposition}[section]
\newtheorem{lemma}{Lemma}[section]
\newtheorem{definition}{Definition}[section]
\newtheorem{example}{Example}[section]
\newtheorem{remark}{Remark}[section]
\newtheorem{solution}{Solution}[section]
\newtheorem{case}{Case}[section]
\newtheorem{condition}{Condition}[section]
\newtheorem{defn}{Definition}[section]
\newtheorem{eg}{Example}[section]
\newtheorem{thm}{Theorem}[section]
\newtheorem{lem}{Lemma}[section]
\newtheorem{soln}{Solution}[section]
\newtheorem{propn}{Proposition}[section]
\newtheorem{ex}{Exercise}[section]
\newtheorem{conj}{Conjecture}[section]
\newtheorem{pb}{Problem}[section]
\newtheorem{cor}{Corollary}[section]
\newtheorem{rmk}{Remark}[section]
\newtheorem{note}{Note}[section]
\newtheorem{notes}{Notes}[section]
\newtheorem{readingex}{Reading exercise}[section]
\newcommand{\pf}{\noindent {\bf Proof:} }
\newcommand{\proof}{\noindent {\it Proof:} }
\frenchspacing

\tikzstyle{level 1}=[level distance=2.75cm, sibling distance=5.65cm]
\tikzstyle{level 2}=[level distance=3cm, sibling distance=2.75cm]
\tikzstyle{level 3}=[level distance=3.9cm, sibling distance=1.5cm]

\tikzstyle{bag} = [text width=10em, text centered] 
\tikzstyle{end} = [circle, minimum width=3pt,fill, inner sep=0pt]

\title{\bf Linear Time Recognition of Equimatchable Split Graphs}
\author{Mehmet Akif Y{\i}ld{\i}z\footnote{Department of Mathematics, Bo\u{g}azi\c{c}i University, 34342, Bebek, Istanbul, Turkey} }



\maketitle
 
\begin{abstract}
\noindent A maximal matching $M$ that consists of independent edges is a subgraph of a simple and undirected graph $G$ for which $G-M$ forms an independent set. A graph $G$ is called equimatchable if all maximal matchings have the same number of edges. On the other hand, $G$ is called as a split graph if its vertices can be partitioned into two subsets for which one of them forms a clique whereas the second forms an independent set. We will give a linear time algorithm for recognition of equimatchable split graphs.
\bigskip

\end{abstract}

\section{Introduction}

Let $G=(V(G),E(G))$ be simple and undirected graph. An induced subgraph $H\subseteq G$ is a graph on a vertex set $V(H)\subseteq V(G)$ such that $uv\in E(H)\Leftrightarrow uv\in E(G)$ for all $u,v\in V(H)$. For an induced subgraph $H$ of $G$, we use the notation $G-H$ to mean the induced subgraph on the vertex set $V(G)-V(H)$. For a vertex $v\in V(G)$, we denote by $N(v)$ the set of \textit{neighbors} of $v$, that is, vertices adjacent to $v$. Also we denote the intersection $N(v)\cap V(H)$ by $N_H(v)$. The \textit{degree} of $v$ is defined as the cardinality of $N(v)$, i.e. $d(v)=|N(v)|$. A vertex $v$ is said to be \textit{isolated} if $d(v)=0$. A set of vertices is called as a \textit{clique} if all vertices in it are pairwise adjacent whereas it is called as an \textit{independent set} if those are pairwise non-adjacent.\\

\noindent A \textit{matching} $M$ is defined as a collection of some edges from $G$ for which any two edges in $M$ has no common end vertices. $M$ is called \textit{maximum} if its cardinality is greater than or equal to all possible matchings of $G$ whereas $M$ is called \textit{maximal} if $G-M$ produces an independent set. From the definition, every maximum matching is also maximal. In the literature \cite{maximum-matching}, there is a well-known polynomial time algorithm to find the cardinality of the maximum matching. In other words, maximum size of maximal matchings can be found in polynomial time. However, in \cite{difficultinbipartite}, authors showed that calculating the minimum cardinality on maximal matchings is NP-hard even for $3$-regular bipartite graphs as an extension of Yannanakis and Gavril's work \cite{yannakakis}.\\

\noindent Examining some specific cases leading to a polynomial solution can be seen as a first step when the general case is NP-hard. Especially, constraints providing an exact solution with a greedy approach are considerable because their execution would be pretty easy. From this perspective, graphs that give a polynomial algorithm for minimum maximal matching have been concerned and the emergence of \textit{equimatchable} graphs was originated from this idea. A graph $G$ is called equimatchable if all maximal matchings have same cardinality. These graphs are first considered in 1974, simultaneously in \cite{equimatchable1974-1}, \cite{equimatchable1974-2}, \cite{equimatchable1974-3}, and formally introduced by Lesk, Plummer and Pulleyblank in 1983 \cite{equimatchablefirst}. Since a maximal matching can be obtained greedily, the matching number becomes equivalent to the size of minimum maximal matching and so the problem is polynomial in those graphs.\\

\noindent In the literature, there are numerous works to recognize whether a given graph belongs to a specific class such as \cite{recognitionofsplit} and \cite{recognitioncograph}. Especially, algorithms having linear running time are quite important when their executions are concerned. Due to study of Demange and Ekim \cite{recognitionofequimatchable}, equimatchable graphs can be recognized in $O(n^2m)$ time where $n$ and $m$ represents to number of vertices and edges in the graph. Although there is no hope in general case, graph classes which lead to obtain a linear time recognition algorithm can be questioned. In this paper, we will examine a graph class for which recognition of equimatchable graphs is linear. A graph $G$ is called \textit{split} if its vertices can be partitioned into two sets where one of them is a clique and the second forms an independent set. Split graphs were first studied by F\"{o}ldes and Hammer in \cite{split1}, \cite{split2}, and independently introduced by Tyshkevich and Chernyak in \cite{split3}. Certifying of split graphs take linear time from \cite{recognitionofsplit}. Thus, equimatchable split graphs can be seen as an ideal candidate to search for a linear time recognition algorithm.

\section{Properties of Equimatchable Split Graphs}
Since isolated vertices are trivially not included in any matching, we will assume there are no isolated vertices in all studied graphs from now on. Throughout the section, let $G=(V,E)$ be a split graph with split partition $(K,I)$ where $K$ is a clique and $I$ is an independent set. If there exists $k\in K$ such that $N_I(k)=\emptyset$, then $K^{'}=K\backslash\{k\}$, $I^{'}=I\cup\{k\}$ will give another split partition for $G$. On the other hand, each $i\in I$ has at least one neighbor in $K$ because there is no isolated vertex in $G$. By this convention, we will have $N_I(k)\neq\emptyset$ for all $k\in K$ and $N(i)\neq\emptyset$ for all $i\in I$. Note that if $S$ is a subgraph of $G$ with $|S|>|I|$, then $S$ has to have either at least two elements from $K$ or exactly one element from $K$ with all elements in $I$. Thus, $S$ has at least one edge, which implies $I$ becomes a maximum independent set.

\begin{lemma}\label{lem:trivialcases}
	If $|K|=1$ or $|I|=1$, then $G$ is equimatchable.
\end{lemma}

\begin{proof}
	Note that $|K|=1$ implies there exists a vertex $v\in V$ such that $N(v)=V-\{v\}$ and $N(u)=\{v\}$ for all $u\in V-\{v\}$. Thus, all maximal matchings contains a single edge. On the other hand, $|I|=1$ implies $uv\in E$ for all $u,v\in V$ and so all maximal matchings have either all vertices of $G$, or miss only one vertex from $G$ depending only on the cardinality of $|V|$. As a result, $|K|=1$ or $|I|=1$ implies $G$ is equimatchable. \qed 
\end{proof}

\begin{lemma}\label{lem:independent-two}
	If $|I|=2$ and $|K|$ is odd, then $G$ is equimatchable.
\end{lemma}

\begin{proof}
	If $|K|=1$, the claim follows from Lemma \ref{lem:trivialcases}. Let $|K|=2p-1$ with $p\geq2$. Note that $G$ has $2p+1$ vertices and we will show that every maximal matching in $G$ has exactly $p$ edges. Assume the contrary, say there exists a maximal matching $M$ contains at most $p-1$ edges. Note that there can be at most $2p-2$ vertices in $M$, but $|V- V(M)|\geq3$ gives there exists an independent set of size at least $3$. However, this is impossible since $I$ is a maximum independent set and it contains only two elements. \qed
\end{proof}

\begin{lemma}\label{lem:cliquepartisodd}
If $G$ is equimatchable with $|I|,|K|\geq 2$, then $|K|$ is odd. 
\end{lemma}

\begin{proof}
Assume the contrary, suppose $G$ is equimatchable, $|I|,|K|\geq 2$ and $|K|=2p$ for some natural number $p$. Since $I$ is an independent set, $K$ gives a maximal matching consisting of $p$ edges. Take $u\in K$ and choose an element $a\in I$ such that $ua\in E$ by using $N_I(u)\neq\emptyset$. Suppose there is an edge between $K-u$ and $I-a$, i.e. take $vb\in E$ for some $v\in K-u$, $b\in I-a$. Since $K- \{u,v\}$ is a clique of size $2p-2$, it has a matching consisting of $p-1$ edges and we can add $ua$ and $vb$ into this matching, which yields a matching consisting of $p+1$ edges as a contradiction to being $G$ equimatchable. Hence, by assuming there are no edges between $K-u$ and $I-a$, we get $bu\in E$ for all $b\in I-a$ since $N(b)\neq\emptyset$. Therefore, we have $N_I(u)=I$. Since $u$ was arbitrary, it follows that $ki\in E$ for all $k\in K$ and $i\in I$. Thus, take $k_1,k_2\in K$ and $i_1,i_2\in I$ by using $|I|,|K|\geq 2$. Observe that $K-\{k_1,k_2\}$ is a clique of size $2p-2$ and so it has a matching consisting of $p-1$ edges. Since we can add $k_1i_1$ and $k_2i_2$ into this matching, we get a contradiction again. As a result, $|K|$ must be odd if $|I|,|K|\geq 2$ and $G$ is equimatchable. \qed
\end{proof}

\begin{lemma}\label{lem:threeindependent}
If $G$ is equimatchable, then there are no six different vertices which satisfy $k_1,k_2,k_3\in K$, $i_1,i_2,i_3\in I$ with $k_1i_1,k_2i_2,k_3i_3\in E$. 
\end{lemma}

\begin{proof}
Assume the contrary, let $k_1,k_2,k_3\in K$, $i_1,i_2,i_3\in I$ with $k_1i_1,k_2i_2,k_3i_3\in E$. Note that $|K|$ is odd from Lemma \ref{lem:cliquepartisodd}, say $|K|=2p-1$ for some natural number $p$. Since $K$ has three different vertices, we get $p\geq 2$. Since $K-\{k_1\}$ is a clique of size $2p-2$, it has $p-1$ independent edges. Then, we can build a maximal matching that consist of exactly $p$ edges by adding $k_1i_1$ into these $p-1$ edges. On the other hand, we can build a maximal matching that consist of $p+1$ edges by taking the edges $k_1i_1$, $k_2i_2$, $k_3i_3$ with $p-2$ independent edges from the clique $K-\{k_1,k_2,k_3\}$ of size $2p-4$, which is a contradiction. \qed
\end{proof}

\begin{lemma}\label{lem:complementing}
Assume $G$ is equimatchable with $|I|\geq3$, $|K|\geq 2$. If $vi\notin E$ for some $v\in K$ and $i\in I$, then either $N(j)=K-N(i)$ or $N(j)=K$ for all $j\in N_I(v)$.
\end{lemma}

\begin{proof}
Let us take $v\in K$ and $i\in I$ with $vi\notin E$. Let $j\in N_I(v)$ and assume that $N(j)\neq K$. Since $N(i)\neq\emptyset$, there exists $u\in K$ with $ui\in E$. If there exists an edge $wk$ between $K-\{u,v\}$ and $I- \{i,j\}$ for some $w\in K-\{u,v\}$ and $k\in I- \{i,j\}$, then the edges $ui$, $vj$, $wk$ contradict with Lemma \ref{lem:threeindependent}. It follows $N_I(w)\subseteq \{i,j\}$ for all $w\in K- \{u,v\}$ and $N(k)\subseteq \{u,v\}$ for all $k\in I- \{i,j\}$. Thus, we get $N(i)\cup N(j)=K$ by noting $v\in N(j)$ and $u\in N(i)$.\\

\noindent We will also prove that $N(i)\cap N(j)=\emptyset$. Assume the contrary, take $w\in N(i)\cap N(j)$. Firstly, observe that $w\neq v$ since $vi\notin E$, and let us first consider the case $w\neq u$. Take an arbitrary element $k\in I-\{i,j\}$. If $k$ is adjacent to $u$ (resp. $v$), the edges $wi$, $ku$, $vj$ (resp. $wi$, $kv$, $uj$) contradict with Lemma \ref{lem:threeindependent}. On the other hand, we have $N(i)- N(j)\neq\emptyset$ since $N(j)\neq K$. If $w=u$, then we get $u,v\notin N(i)- N(j)$ and so we can find $t\in K-\{u,v\}$ such that $ti\in K$. Thus, take an arbitrary element $k\in I-\{i,j\}$. Similarly, if $k$ is adjacent to $u$ (resp. $v$), the edges $ti$, $ku$, $vj$ (resp. $ti$, $kv$, $uj$) contradict with Lemma \ref{lem:threeindependent}. \qed \\
  
\end{proof}

\begin{lemma}\label{lem:generalcase}
If $G$ is equimatchable with $|I|\geq3$, $|K|\geq 2$, then there exist $x\in K$ and $y\in I$ such that $N(z)=\{x\}$ for all $z\in I-\{y\}$, and either $N(y)=K-\{x\}$ or $N(y)=K$. 
\end{lemma}

\begin{proof}
Assume $G$ is equimatchable with $|I|\geq3$, $|K|\geq 2$. Take a vertex $y\in I$ such that $d(y)\geq d(j)$ for all $j\in I$. Also, note that $|K|\geq 3$ since $|K|$ is odd from Lemma \ref{lem:cliquepartisodd}.\\

\noindent Suppose there exists a vertex $x\in K$ with $xy\notin E$. We claim $N_I(x)=I-\{y\}$, and assume the contrary. Let us take $t\in I-\{y\}$ with $xt\notin E$. Since $N_I(x)\neq\emptyset$, there exists $s\in I$ such that $xs\in E$. Note that $xy\notin E$ and $xs\in E$ imply $N(s)=K-N(y)$ or $N(s)=K$ from Lemma \ref{lem:complementing}. Since $d(y)\geq d(s)$ and $N(y)\neq K$, we get $N(s)=K-N(y)$. Similarly, $xt\notin E$ and $xs\in E$ imply $N(s)=K-N(t)$ or $N(s)=K$ from Lemma \ref{lem:complementing}. Thus, we get $N(s)=K-N(t)$ and so $N(y)=N(t)$. On the other hand, $xy\notin E$, $xs\in E$ and $d(y)\geq d(s)$ give there exists $a\in K$ such that $ay\in E$ and $as\notin E$. On the other hand, note that each vertex in $K-\{x,a\}$ is adjacent to exactly one of $y$ and $s$. Since $d(y)\geq d(s)$ and $|K|\geq 3$, we can find $b\in K-\{x,a\}$ such that $by\in E$. Thus, the edges $at$, $xs$, $by$ contradict with Lemma \ref{lem:threeindependent}. As a result, we have $N_I(x)=I-\{y\}$. Again, by using Lemma \ref{lem:complementing}, we get either $N(z)=K-N(y)$ or $N(z)=K$ for all $z\in I-\{y\}$. Since $d(y)\geq d(z)$ for all $z\in I$ and $N(y)\neq K$, we can conclude that if $N(y)\neq K$, then $N(z)=K-N(y)$ for all $z\in I-\{y\}$. Now, if there exist a vertex in $K-N(y)$ other than $x$, say $x_1$, let us take $z_1,z_2\in I-\{y\}$ by using $|I|\geq 3$. Again, the edges $xz_1$, $x_1z_2$, $ya$ contradict with Lemma \ref{lem:threeindependent}. Therefore, we have $N(y)=K-\{x\}$ and $N(z)=\{x\}$ for all $z\in I-\{x\}$ whenever $N(y)\neq K$.\\

\noindent Suppose $N(y)=K$, and take a vertex $z\in I-\{y\}$. Since $N(z)\neq\emptyset$, there exists $x\in K$ with $xz\in E$. Suppose $z$ has a neighbor other than $x$, say $x_1$. By using $|I|\geq 3$, choose a vertex $z_1\in I-\{y,z\}$. If $z_1$ has a different neighbor other than $x$ and $x_1$, say $x_2$, then the edges $yx$, $zx_1$, $z_1x_2$ contradict with Lemma \ref{lem:threeindependent}. Thus, $z_1$ is adjacent to at least one of $x$ and $x_1$. Also, by using $|K|\geq 3$ and $N(y)=K$, take $u\in K-\{x,x_1\}$ with $yu\in E$. If $z_1$ is adjacent to $x$ (resp. $x_1$), then the edges $z_1x$, $zx_1$, $yu$ (resp. $z_1x_1$, $zx$, $yu$) contradict with Lemma \ref{lem:threeindependent}. As a result, the only neighbor of $z$ becomes $x$. Now, take an element $k\in I-\{y,z\}$. Similarly, $k$ has a unique neighbor, say $v$. If $v$ is different from $x$, take an element from $w\in K-\{x,v\}$. Since $N(y)=K$, the edges $yw$, $zx$, $vk$ contradict with Lemma \ref{lem:threeindependent}. Therefore, $x$ becomes the only neighbor of each vertex in $I-\{y\}$, which completes the proof. \qed 
\end{proof}
            

\section{Characterization of Equimatchable Split Graphs}

In this section, we will give the characterization of equimatchable split graphs.

\begin{theorem}\label{thm:characterization}
Let $G=(V,E)$ be a simple undirected graph on $n\geq4$ vertices with no isolated vertices. Let $r$ and $p$ be the number of vertices of degree $1$ and $n-1$, respectively. Then, $G$ is an equimatchable split graph if and only if one of the followings holds:
\begin{itemize}
\item[(i)] $p=n$.
\item[(ii)] $r=n-1$ and $p=1$.
\item[(iii)] $p=1$, $r\geq 2$, $n-r$ is even, and all vertices have degree $1$, $n-r-1$ or $n-1$.
\item[(iv)] $p=0$, $r\geq 2$, $n-r$ is even, there are two vertices $x$ and $y$ with $xy\notin E$ such that $d(x)=n-2$, $d(y)=n-r-2$, and all vertices in $V-\{x,y\}$ have degree $1$ or $n-r-1$.
\item[(v)] There are two vertices $x$ and $y$ such that $n$ is odd, $d(x)+d(y)=p+n-2$ and all vertices in $V-\{x,y\}$ have degree $n-1$ or $n-2$. 
\end{itemize}
\end{theorem}

\begin{proof}
Firstly, let $G$ be an equimatchable split graph with split partition $(K,I)$ where $K$ is a clique and $I$ is an independent set. As similar to the previous section, we can assume $N_I(k)\neq\emptyset$ for all $k\in K$ and $N(i)\neq\emptyset$ for all $i\in I$. Observe that if $|I|=1$, then $G$ becomes a complete graph and so (i) is satisfied. Similarly, if $|K|=1$, then $G$ becomes a star and so (ii) is satisfied. Assume $|I|,|K|\geq 2$, we get $|K|$ is odd from Lemma \ref{lem:cliquepartisodd}.\\

\noindent Now, if $|I|=2$, let us take $x,y\in I$. By using $|K|$ is odd, we get $n$ is odd. Since each element in $K$ is adjacent to at least one of $x$ and $y$, each vertex in $K$ has degree $n-2$ or $n-1$. Moreover, $d(x)+d(y)-|K|$ is equal to the number of vertices in $K$ which are adjacent to both of $x$ and $y$, in other words number of vertices of degree $n-1$. Thus, we get $d(x)+d(y)=p+|K|=p+n-2$ and so (v) is satisfied.\\

\noindent Hence, let us assume $|I|\geq 3$ and $|K|\geq 2$. From Lemma \ref{lem:generalcase}, there exist $x\in K$ and $y\in I$ such that $N(z)=\{x\}$ for all $z\in I-\{y\}$, and either $N(y)=K-\{x\}$ or $N(y)=K$. Since $|K|$ is odd from Lemma \ref{lem:cliquepartisodd}, we have $|K|\geq 3$. Then, all vertices in $I-\{y\}$ have degree $1$ whereas all vertices in $K\cup \{y\}$ have degree at least two. Thus, $|K|=n-|I|=n-r-1$ implies $n-r$ is even. On the other hand, if $N(y)=K$, then we get $d(x)=n-1$ and $K\cup \{y\}$ becomes a clique of size $n-r$. As a result, all vertices in $K\cup \{y\}$ except $x$ have degree $n-r-1$, then (iii) is satisfied. Similarly, if $N(y)=K-\{x\}$, clearly (iv) is satisfied. Therefore, if $G$ is an equimatchable split graph, then at least one of (i)-(v) is satisfied.\\

\noindent Conversely, let $G=(V,E)$ be a simple undirected graph on $n\geq 4$ vertices. If (i) holds, then $G$ is a clique and if (ii) holds, then $G$ is a star. In both cases, $G$ becomes a split graph where clique or independent part has only one vertex, thus the result follows from Lemma \ref{lem:trivialcases}.\\

\noindent Assume (iii) holds. Let $u$ be the unique vertex of degree $n-1$ and $I$ be the set of vertices of degree $1$. Note that there is no edge between $I$ and $V-u$. On the other hand, each vertex in $V-(I\cup\{ u\})$ has degree $n-r-1$ with $|V-(I\cup \{u\})|=n-r-1$. Thus, $(V-I,I)$ becomes a split partition for $G$. Observe that for any maximal matching $M$, it can have at most one vertex from $I$ since each vertex in $I$ has degree $1$ and all of them are adjacent to $u$. Moreover, if there exists a vertex $w\in I\cap V(M)$, then we get $uw\in E(M)$. Since $V-(I\cup\{u\})$ is a clique and $|V-(I\cup\{u\})|=n-r-1$ is odd, we get $M$ has $\dfrac{n-r-2}{2}+1=\dfrac{n-r}{2}$ edges. Similarly, if $M$ has no vertices from $I$, then $M$ has $\dfrac{n-r}{2}$ edges since $V-I$ is a clique. As a result, all maximal matchings in $G$ has $\dfrac{n-r}{2}$ edges and so $G$ is equimatchable.\\

\noindent Assume (iv) holds, and take the vertices $x$ and $y$ with $d(x)=n-2$, $d(y)=n-r-2$, $xy\notin E$. Similarly, let $I$ be the set of vertices of degree $1$. Note that $x$ is the unique neighbor of the vertices in $I$. Thus, $I\cup \{y\}$ becomes an independent set. Also, $d(y)=n-r-2$ and $|(V-I)-\{x,y\}|=n-r-2$ implies $y$ is adjacent to all vertices in $(V-I)-\{x,y\}$. Moreover, each vertex in $(V-I)-\{x,y\}$ has degree $n-r-1$, and each of them is not adjacent to any vertex in $I$. Thus, $V-(I\cup\{y\})$ becomes a clique, so we get $G$ is a split graph. Take a maximal matching $M$, we claim $M$ has exactly $\dfrac{n-r}{2}$ edges. Firstly, it can have at most one vertex from $I$ since each vertex in $I$ has degree $1$ and all of them are adjacent to $x$. Thus, $M$ can have at most $\dfrac{n-r+1}{2}$ edges, and by using the fact that $n-r$ is even, we get there are at most $\dfrac{n-r}{2}$ edges in $M$. Also, $M$ can miss at most one vertex from $V-(I\cup\{y\})$ since it is a clique. Suppose $t\in V-(I\cup\{y\})$ but $t\notin M$. Note that $t\neq x$ since $x$ is the unique neighbor of the vertices in $I$. Thus, $t\notin M$ implies $y\in M$. As a result, we get $|V(M)|\geq |V-(I\cup\{y\})|=n-r-1$. Since $n-r$ is even, we get $M$ has exactly $\dfrac{n-r}{2}$ edges and so $G$ is equimatchable.\\

\noindent Assume (v) holds, and take the vertices $x$ and $y$ with $d(x)+d(y)=p+n-2$. Let $A$ be the set of vertices of degree $n-1$, and $B=(V-\{x,y\})-A$. Note that each vertex in $B$ has degree $n-2$. Let us define $A_x$ (resp. $A_y$) as the set of vertices in $B$ which miss $x$ (resp. $y$). It is clear that $A_x$, $A_y$ and $A$ are disjoint, and observe that
$$d(x)=n-1-|A_x|-\mathbbm{1}_{\{xy\notin E\}}\text{ and }d(y)=n-1-|A_y|-\mathbbm{1}_{\{xy\notin E\}}$$

\noindent Thus, $p+n-2=2n-2-|A_x|-|A_y|-2\cdot \mathbbm{1}_{\{xy\notin E\}}$ implies $|A|+|A_x|+|A_y|+2\cdot \mathbbm{1}_{\{xy\notin E\}}=n$. Since $A$, $A_x$, $A_y$ are disjoint, we have $2\cdot \mathbbm{1}_{\{xy\notin E\}}=|V-(A\cup A_x\cup A_y)|$.\\

\noindent Suppose $xy\notin E$. Then, we have $x,y\notin A$ and so $x,y\in V-(A\cup A_x\cup A_y)$. Thus, this equality can be true only if $V-(A\cup A_x\cup A_y)=\{x,y\}$. In other words, $xy\notin E$ implies $B=A_x\cup A_y$. As a result, $A\cup B=V-\{x,y\}$ forms a clique and so $G$ becomes a split graph with split partition $(A\cup B,\{x,y\})$. Then, $G$ is equimatchable from Lemma \ref{lem:independent-two} because each vertex in $A\cup B$ has a neighbor in $\{x,y\}$ and $|A\cup B|=n-2$ is odd.\\

\noindent Suppose $xy\in E$, in other words, assume $V-(A\cup A_x\cup A_y)=\emptyset$. Thus, $x,y\notin V-(A\cup A_x\cup A_y)$ gives $x,y\in A$, which implies $d(x)=d(y)=n-1$. By using $p+n-2=d(x)+d(y)$, we can conclude that $p=n$ and so (i) also holds. As a result, if $G$ satisfies at least one of (i)-(v), then it an equimatchable split graph. \qed   

\end{proof}


\section{Recognition Algorithm}

Let $G=(V,E)$ be an undirected simple graph of order $n\geq 4$ with no isolated vertices, and take a non-decreasing degree ordering $(v_1,v_2,...,v_n)$ of $G$. Define $r$ and $p$ as in the Theorem \ref{thm:characterization}. We can observe the followings:
\begin{enumerate}
\item If $d(v_2)=n-1$, then each vertex $v_k$ is adjacent to all remaining vertices for $k\geq 2$, which simply implies $d(v_1)=n-1$ and so $p=n$.
\item Suppose $G$ is an equimatchable split graph. If $n-2\geq d(v_2)\geq 2$, then we get $p\leq n-2$ and $r\leq 1$. Thus, $G$ must satisfy the condition (v) of Theorem \ref{thm:characterization}. Take $x,y\in V$ with $d(x)+d(y)=p+n-2$ and $d(x)\leq d(y)$ where each vertex in $V-\{x,y\}$ has degree at least $n-2$. Thus, we get $d(v_3)\geq n-2\geq d(x)$, which implies $x$ has exactly $p$ neighbors of degree $n-1$ and so $d(x)\geq p$. If $d(x)=p$, we get $d(y)=n-2$, which gives $d(v_2)=n-2$ and so $d(v_1)+d(v_2)=p+n-2$. If $d(x)\geq p+1$, we have $d(y)\leq n-3$, which implies $\{x,y\}=\{v_1,v_2\}$ since $d(v_3)\geq n-2$, again we get $d(v_1)+d(v_2)=p+n-2$.
\item Suppose $G$ is an equimatchable split graph with $d(v_2)=1$. Clearly, conditions (i) and (v) in Theorem \ref{thm:characterization} cannot be satisfied. Thus, at least one of (ii)-(iii)-(iv) must be satisfied and so we have $p\leq 1$. Also, the inequality $r\geq 2$ of the conditions (iii)-(iv) can be omitted. 
\end{enumerate}
 
\noindent With these observations, it is clear that Algorithm \ref{algo:equisplit} works. On the other hand, it is well-known that a non-decreasing degree sequence of a graph $G$ can be obtained in $O(m+n)$ time where $m$ and $n$ represents the number of edges and vertices of $G$, respectively. Moreover, steps (2)-(3)-(4) will take $O(n)$ time since they only need to check the degrees of the vertices. For the rest, each step requires a constant time. As a result, Algorithm \ref{algo:equisplit} has a linear running time and answers whether a given graph $G$ is an equimatchable split graph or not.

\vspace{0.4 cm}

\noindent \textbf{Acknowledgements:} I would like to thank T{\i}naz Ekim (Assoc. Prof. at Bo\u{g}azi\c{c}i University) for introducing me equimatchable graphs and giving me this problem. Also, she stated that the problem is originated from some private communications with Nina Chiarelli (Asst. Prof. at University of Primorska), Martin Milanic (Prof. at University of Primorska) and Didem G\"{o}z\"{u}pek (Assoc. Prof. at Gebze Technical University).


\begin{algorithm}[H]
	
	\KwInput{An undirected simple graph $G=(V,E)$ of order $n\geq 4$ with no isolated vertices.}
	\KwOutput{A reply \textbf{YES} if $G$ is an equimatchable split graph, and a reply \textbf{NO} otherwise.}
	Compute a non-decreasing degree ordering $(v_1,v_2,...,v_n)$ of $G$.\\
	Set $p$ as the number of vertices of degree $n-1$.\\
	Set $r$ as the number of vertices of degree $1$.\\
	Set $q$ as the number of vertices of degree $n-r-1$.\\
	\If{$d(v_2)=n-1$}{
	output \textbf{YES};\\
	\textbf{return};}

	\If{$d(v_2)\geq 2$}{
		\If{$n$ is odd and $d(v_3)\geq n-2$ and $d(v_1)+d(v_2)= p+n-2$}
			{output \textbf{YES};\\
			\textbf{return};}
		
		
		
	output \textbf{NO};\\
	\textbf{return};}

		
		
		

		\If{$p=1$}{
		\If{$r=n-1$}
		{output \textbf{YES};\\
			\textbf{return};}
		\If{$n-r$ is even and $q= n-r-1$}{output \textbf{YES};\\
				\textbf{return};}

		}
			\If{$n-r$ is even and $d(v_{r+1})=q= n-r-2$ and $d(v_n)= n-2$ and $v_{r+1}v_{n}\notin E$}{output \textbf{YES};\\
				\textbf{return};}
	output \textbf{NO};\\
		
	\caption{EquiSplit} \label{algo:equisplit}
\end{algorithm}




\begin{thebibliography}{1}
	
	\bibitem{equimatchable1974-3} B. Gr\"{u}nbaum., "Matchings in polytopal graphs", \textit{Networks}, 4:175-190, 1974.
	
	\bibitem{recognitioncograph} D. G. Corneil, Y. Perl, L. K. Stewart, "A linear recognition algorithm for cographs", \textit{SIAM J. Comput.}, 14(4):926–934, 1985.
	
	\bibitem{equimatchable1974-2} D. H.-C. Meng., \textit{Matchings and coverings for graphs}, PhD Thesis, Michigan State University, East Lansing, MI, 1974.
	
	\bibitem{maximum-matching} J. Edmonds, "Paths, trees, and flowers", \textit{Canadian Journal of mathematics}, 17:449-467, 1965.
	
	\bibitem{recognitionofequimatchable} M. Demange, T. Ekim, "Efficient recognition of equimatchable graphs", 2003.	
	
	\bibitem{difficultinbipartite} M. Demange, T. Ekim, "Minimum Maximal Matching is NP-hard in Regular Bipartite Graphs", 2008.
	
	\bibitem{equimatchablefirst} M. Lesk, M.D. Plummer, W.R. Pulleyblank, "Equimatchable graphs", \textit{Graph Theory and Combinatorics}, Cambridge, 1983.
	
	\bibitem{equimatchable1974-1} M. Lewin, "Matching-perfect and cover-perfect graphs", \textit{Israel Journal of Mathematics}, 18:345-347, 1974.
	
	\bibitem{yannakakis} M. Yannakakis, F. Gavril, "Edge dominating sets in graphs", \textit{SIAM J. Appl. Math.}, 38:364-372, 1980.
	
	\bibitem{recognitionofsplit} P. Heggernes, D. Kratsch, "Linear-time certifying algorithms for recognizing split graphs and related graph classes", \textit{Nordic Journal of Computing}, 14(1):87-108, 2007
	
	\bibitem{split3} R. I. Tyshkevich, A. A. Chernyak, "Canonical partition of a graph defined by the degrees of its vertices", \textit{Isv. Akad. Nauk BSSR, Ser. Fiz.-Mat. Nauk (in Russian)}, 5: 14-26.	
	
	\bibitem{split1} S. F\"{o}ldes, P. L. Hammer, "Split Graphs", \textit{Proceedings of the Eighth Southeastern Conference on Combinatorics, Graph Theory and Computing, Louisiana State Univ., Baton Rouge, La., Congressus Numerantium}, XIX, Winnipeg: Utilitas Math., pp. 311-315, 1977.
	
	\bibitem{split2} S. F\"{o}ldes, P. L. Hammer, "Split graphs having Dilworth number two", \textit{Canadian Journal of Mathematics}, 29(3):666-672.
	
	
\end{thebibliography}
\end{document}